\documentclass[12pt,draft]{article}

\usepackage{amsmath,amssymb,amsthm,xspace,amscd}

\theoremstyle{remark}
\newtheorem{remark}{Remark}

\theoremstyle{plain}
\newtheorem{theorem}{Theorem}

\newtheorem{proposition}{Proposition}
\newtheorem{lemma}{Lemma}

\newcommand{\be}{{\mathrm b}}
\newcommand{\en}{{\mathrm e}}

\newcommand{\A}{{\mathcal A}}

\newcommand{\R}{{\mathbb R}}
\newcommand{\Z}{{\mathbb Z}}
\newcommand{\Q}{{\mathbb Q}}
\newcommand{\eps}{{\varepsilon}}
\newcommand{\vect}{{\mathbf v}}
\newcommand{\wect}{{\mathbf w}}
\newcommand{\uect}{{\mathbf u}}
\newcommand{\xect}{{\mathbf x}}
\newcommand{\yect}{{\mathbf y}}

\begin{document}

\title{Rotational Interval Exchange Transformations}
\author{Alexey TEPLINSKY
\thanks{Institute of Mathematics, Natl.\ Acad.\ Sc.\ Ukraine; e-mail:\ teplinsky.a@gmail.com}}
\maketitle
\begin{abstract}
We show the equivalence of two possible definitions of a rotational interval exchange transformation: by the first one, it is a first return map for a circle rotation onto a union of finite number of circle arcs, and by the second one, it is an interval exchange with a scheme (in the sense of interval rearrangement ensembles), whose dual is an interval exchange scheme as well.
\end{abstract}

\section{Introduction}
\label{sect:introduction}

In the paper~\cite{Teplinsky23} we proposed a new concept of interval rearrangement ensembles (IREs) that generalizes the classical for dynamical systems construction of interval exchange transformations~\cite{Keane75,Veech78,Rauzy79,Keane-Rauzy80,Masur82,Veech82}. The cornerstone in our concept is the duality involution in the space of schemes (i.~e., discrete components) of IREs that produces a dual IRE scheme to every given IRE scheme, and this duality reverses time for the Rauzy--Veech type induction. Interval exchange schemes constitute a partial case of IRE schemes and, as it appeared, may be or may be not interval exchange schemes themselves. In a sense, the space of all IRE schemes is the extension of the space of all (multi-segment) interval exchange schemes due to the duality operation. On the other hand, in the space of all interval exchange schemes, there is a subspace consisting of those whose duals are interval exchange schemes as well. Interval exchange transformations with schemes from this subspace we call rotational, because these exchanges are related to circle rotations, namely---they are first return maps for circle rotations onto a union of finite number of arcs. Effectively, we talk about two approaches to defining one and the same object: the first approach is through the duality on IRE schemes, and the second one is through first return maps on a circle. The aim of this work is to describe exactly the relation between these approaches. Our results are formulated as three statements of Theorem~1 and proved in corresponding sections of this paper.

We believe that studying rotational interval exchanges in the framework of IRE concept paves way to new results on still open problems in the rigidity theory for circle diffeomorphisms with multiple breaks, similar to those we obtained earlier for circle diffeomorphisms with a single break in~\cite{Teplinsky08ukr,KhaninTeplinsky13}. This is because the most useful tool for investigating circle rotations with special points is renormalization group approach, which substitutes the initial map by a sequence of first return maps onto unions of small neighborhoods, renormalized from exponentially small to macroscopic lengths. In a sequence of first return maps, the next map is obtained from the previous one by applying induction of Rauzy-Veech type, and therefore a duality allowing to reverse time in this process will be an important tool in further studies.

If special points of an irrational circle diffeomorphism are non-degenerate (i.~e., the left derivative is not equal to the right derivative, but both are positive), then the renormalized first return maps approach certain finite dimensional spaces as the lengths of neighborhoods diminish. These spaces consist of linear-fractional maps, as it was first shown in~\cite{KhaninVul91}; and in the case when the product of all break sizes equals~1, the limit spaces consist of affine maps, as investigated in particular in~\cite{CunhaSmania13,CunhaSmania14}. 

The next logical step will be extending the duality from the space of IRE schemes to the corresponding spaces of linear-fractional maps; we work currently on this task and plan future publications on it.

The structure of this paper is following: in Section~\ref{sect:definitions} we remind basic notions of our theory of interval rearrangement ensembles, in Section~\ref{sect:theorem} we formulate the main result as a theorem of three statements, and then prove the statements 1)--3) of this theorem in Sections~\ref{sect:proof_1}--\ref{sect:proof_3} respectively.

\section{The basic notions on IREs and interval exchanges}
\label{sect:definitions}

Here we recall the main notions in theory of interval rearrangement ensembles, which was presented in~\cite{Teplinsky23}.

Let $\A$ be an alphabet of $d\ge1$ symbols; these will be labels for intervals in our rearrangement ensemble. Consider the \emph{doubled alphabet} $\bar\A=\A\times\{\be,\en\}$ (here the letters ``$\be$'' and  ``$\en$'' come from the words ``beginning'' and ``ending'' respectively) and any permutation $\sigma$ of this doubled alphabet, i.e., a bijective map from $\bar\A$ onto $\bar\A$. We call this permutation a \emph{scheme} of an interval rearrangement ensemble (i.e., an ``IRE scheme''), while an \emph{interval rearrangement ensemble} (i.e., an ``IRE'') itself is a pair $(\sigma,\xect)$, in which a scheme $\sigma$ is equipped with a \emph{vector of endpoints} $\xect\in\R^{\bar\A}$, whose coordinates satisfy the equalities
\begin{equation}\label{eq:endpoints_relation}
x_{\alpha\be}+x_{\alpha\en}-x_{\sigma(\alpha\be)}+x_{\sigma(\alpha\en)}=0\quad{\text{for all}}\quad\alpha\in\A.
\end{equation} 
A vector $\xect$ satisfying (\ref{eq:endpoints_relation}) is called \emph{allowed} by the scheme $\sigma$.
For a given IRE $(\sigma,\xect)$, the \emph{vector of lengths} $\vect\in\R^{\A}$ is defined coordinate-wise as
\begin{equation}\label{eq:lengths_relation}
v_\alpha=x_{\sigma(\alpha\be)}-x_{\alpha\be}=x_{\alpha\en}-x_{\sigma(\alpha\en)},\quad\alpha\in\A.
\end{equation}

We call two IREs \emph{shift equivalent}, if both their schemes and their vectors of lengths are the same.
A vector $\vect\in\R^{\A}$ is called a vector of lengths \emph{allowed} by the scheme $\sigma$, if there exists a vector of endpoints allowed by that scheme and satisfying (\ref{eq:lengths_relation}). A pair $(\sigma,\vect)$, in which the vector of lengths $\vect$ is allowed by the scheme $\sigma$, we call a \emph{floating IRE}, to distinguish it from a ``fixed'' IRE $(\sigma,\xect)$. A floating IRE can be seen as an equivalence class of shift equivalent fixed IREs.

In this paper, we mostly work with floating IREs, and regularly use the following simple criterion on whether a vector of lengths is allowed by a given scheme. The key fact here is that a scheme $\sigma$, as a permutation, can be decomposed into $N\ge1$ disjoint cycles of the form $c=c(\bar\xi)=(\bar\xi,\sigma(\bar\xi),\dots,\sigma^k(\bar\xi))$, $\bar\xi\in\bar\A$, where $k\ge0$, $\sigma^{k+1}(\bar\xi)=\bar\xi$, $\sigma^i(\bar\xi)\ne\bar\xi$ for $0<i\le k$, and the cycles $c(\bar\xi)$ and $c(\sigma^i(\bar\xi))$ are considered the same.

\begin{proposition}\label{prop:v=v}
A vector of lengths $\vect$ is allowed by a scheme $\sigma$, if and only if for every cycle $c=c(\bar\xi)$, $\bar\xi\in\bar\A$, in this scheme, the following equality holds:
\begin{equation}\label{eq:v=v}
\sum_{\alpha:\,\alpha\be\in c}v_\alpha=\sum_{\alpha:\,\alpha\en\in c}v_\alpha.
\end{equation}
\end{proposition}

\begin{proof}
Let the vector of lengths $\vect$ be allowed by the scheme $\sigma$. Then there exists a vector of endpoints $\xect$ such that the equalities (\ref{eq:lengths_relation}) hold. In accordance with them, $x_{\sigma(\bar\eta)}-x_{\bar\eta}=v_\alpha$ for $\bar\eta=\alpha\be$, and $x_{\sigma(\bar\eta)}-x_{\bar\eta}=-v_\alpha$ for $\bar\eta=\alpha\en$. The equalities (\ref{eq:v=v}) follow from the equivalences
$$
\sum_{\bar\eta\in c(\bar\xi)} (x_{\sigma(\bar\eta)}-x_{\bar\eta})=(x_{\sigma(\bar\xi)}-x_{\bar\xi})+(x_{\sigma^2(\bar\xi)}-x_{\sigma(\bar\xi)})+\dots+(x_{\bar\xi}-x_{\sigma^k(\bar\xi)})=0
$$
along each cycle $c(\bar\xi)$.

Now let us assume that the vector $\vect$ satisfies all the equalities (\ref{eq:v=v}). For every cycle $c(\bar\xi)$ in the scheme $\sigma$, take an arbitrary number as a coordinate of the endpoint $x_{\bar\xi}$, and the rest of endpoints determine by the following algorithm: if the value of $x_{\bar\eta}$ is already determined, while the value of $x_{\sigma(\bar\eta)}$ is not yet, then we put $x_{\sigma(\bar\eta)}=x_{\bar\eta}+v_\alpha$ in the case $\bar\eta=\alpha\be$, or $x_{\sigma(\bar\eta)}=x_{\bar\eta}-v_\alpha$ in the case $\bar\eta=\alpha\en$. Due to the equalities (\ref{eq:v=v}), the same relations hold for $\bar\eta=\sigma^k(\bar\xi)$, therefore the vector of endpoints $\xect$ is allowed, and $\vect$ satisfies (\ref{eq:lengths_relation}).

Proposition is proved.
\end{proof}

The cornerstone of our theory of IREs is the notion of duality, which reverses time in application of Rauzy--Veech type of induction to IRE schemes (in Sect.~\ref{subsect:induction} below we remind the definitions of elementary steps of this induction). 
Two IRE schemes $\sigma$ and $\sigma^*$ are called \emph{dual} to each other, if
\begin{equation}\label{eq:duality}
\sigma^*(\alpha\be)=\sigma(\alpha\en),\quad
\sigma^*(\alpha\en)=\sigma(\alpha\be)\qquad\text{for all}\qquad\alpha\in\A.
\end{equation}
Through this duality, we will define rotational interval exchange schemes in the next section.

An IRE scheme is called \emph{irreducible}, if an equality $\sigma(\bar\A_0)=\bar\A_0$, where $\bar\A_0=\A_0\times\{\be,\en\}$, $\A_0\subset\A$, implies $\A_0\in\{\varnothing,\A\}$. It is easy to see from the definition of duality~(\ref{eq:duality}) that $\sigma(\bar\A_0)=\sigma^*(\bar\A_0)$ for any subset $\A_0\subset\A$, hence irreducibility of $\sigma$ implies irreducibility of $\sigma^*$, and vice versa. If a scheme is not irreducible, then an IRE with such a scheme is effectivly decomposed into two or more totally independent IREs, and for that reason, in study of their dynamics, it is reasonable to consider only IREs with irreducible schemes.

It follows from Proposition~2 of~\cite{Teplinsky23} that in the case of irreducible scheme $\sigma$, exactly $N-1$ among $N$ equalities~(\ref{eq:v=v}) are linearly independent (the sum of all equalities~(\ref{eq:v=v}) is a trivial equivalence, therefore any $N-1$ of them are linearly independent).

An IRE is called \emph{positive}, if its vector of lengths is positive. An IRE scheme is called positive, if it allows a positive vector of lengths.

A positive IRE must be visualized as an ensemble of $2d$ intervals, paired in $d$ pairs with labels $\alpha\in\A$; inevery such pair, the \emph{beginning interval} $I_{\alpha\be}=[x_{\alpha\be},x_{\sigma(\alpha\be)})$ and the \emph{ending interval} $I_{\alpha\en}=[x_{\sigma(\alpha\en)},x_{\alpha\en})$ with the same label $\alpha$ have the same length $v_\alpha$. In accodance with the cycles in the scheme, all beginning and ending intervals with corresponding subscripts are connected by their endpoints into $N$ closed chains, i.e., one-dimensional polygonal lines. One may consider such a closed polygonal chain as a path going from $x_{\bar\xi}$ to $x_{\sigma(\bar\xi)}$, then from $x_{\sigma(\bar\xi)}$ to $x_{\sigma^2(\bar\xi)}$, and so on until returning to $x_{\bar\xi}$; on this path, every beginning interval is passed from left to right, and every ending interval is passed from right to left. A parallel translation of any of these $N$ closed one-dimensional polygonal lines as a whole, obviously, does not affect the equalities~(\ref{eq:lengths_relation}) and does not change the lengths of intervals; this is the reason why we call a pair $(\sigma,\vect)$ ``floating IRE'', and the corresponding fixed IRE ``shift equivalent''. 

Now, let us consider a special case when every cycle in the scheme $\sigma$ of a positive IRE can be split into two arcs, one consisting of beginning intervals only, and another of ending intervals only (i.e., every cycle has the form $c=(\alpha_1\be,\dots,\alpha_m\be,\beta_n\en,\dots,\beta_1\en)$ for some labels $\alpha_1,\dots,\alpha_m$, $\beta_1,\dots,\beta_n\in\A$, $n,m\ge1$). It is natural to call such IRE an \emph{interval exchange} and associate with the discrete dynamical system (mapping) on a disjoint union of segments $J_1,\dots,J_N$ corresponding to the cycles $c_1,\dots,c_N$ (namely, a segment $J=[x_{\alpha_1\be},x_{\beta_n\en})=\bigcup_{i=1}^mI_{\alpha_i\be}=\bigcup_{i=1}^nI_{\beta_i\en}$ corresponds to the cycle $c$ given above); this mapping shifts every beginning interval onto the ending interval with the same label from the alphabet $\A$. Proposition~\ref{prop:v=v} in this case says that a vector of lengths is allowed by a scheme, if and only if, for every segment $J$, the sum of lengths of all beginning segments included within in it is equal to the sum of lengths of all ending segments within it. We will be calling by the same name ``interval exchange'' a pair $(\sigma,\xect)$, an associated mapping, and an induced one-dimensional dynamical system. A $\sigma$ in this case is called an \emph{interval exchange scheme}. Similar to general IREs, one may consider \emph{floating interval exchanges} $(\sigma,\vect)$, restricting attention to the lengths of intervals only and allowing segments $J_1,\dots,J_N$ to shift freely along the axis, i.e., factorizing the space of all interval exchanges w.r.t.\ the shift equivalence.

For cycles of the above-mentioned form $(\alpha_1\be,\dots,\alpha_m\be,\beta_n\en,\dots,\beta_1\en)$, it is convenient to use more visual ``two-row notation''
$c=\left[\begin{array}{ccc}
\alpha_1 & \dots & \alpha_m \\
\beta_1 & \dots & \beta_n                                                                                                                                           \end{array}\right]$ (the entire interval exchange scheme can also be written in two rows as the set af all its cycles $\sigma=\{c_1,\dots,c_N\}$, $N\ge1$). 
In terms of~\cite{Teplinsky23}, such a ``two-row'' cycle is called a cycle with zero twist number. To be precise, the \emph{twist number} of a cycle in an IRE scheme $\sigma$ is the number of positions in this cycle, where a beginning element is followed by an ending element, i.e., $\sigma(\alpha\be)=\beta\en$ for some $\alpha,\beta\in\A$, minus one. If a cycle consists either of beginning elements only, or of ending elements only, then its twist number is $-1$, however the condition of positivity makes this case impossible.
The twist number $T(\sigma)$ of a scheme is the sum of the twist numbers of all its cycles. Therefore, in this terminology, an \emph{interval exchange scheme} is a positive IRE scheme with zero twist number, and an \emph{interval exchange} itself is a positive IRE with such a scheme.

The definition above determines a classical interval exchange transformation, if its scheme consists of a single cycle, and the left endpoint of the corresponding segment lies at the origin; these restrictions do not look natural to us, and so we call ``interval exchange'' the more general construction on multiple segments. In Sect.~3 of~\cite{Teplinsky23}, before we defined an IRE, we formulated a classical-like definition for a multi-segment interval exchange transformation. One may see how much simpler is the definition of interval exchange in framework of IRE concept comparing to the classical-like definition.

\section{Rotational schemes and the main result}
\label{sect:theorem}

For an interval exchange scheme $\sigma$, its dual $\sigma^*$ does not have to, but can be an interval exchange scheme as well, and the schemes with this property compose an important special class, which is the subject of this paper. Thus, we call an interval exchange scheme \emph{rotational}, if its dual is an interval exchange scheme as well. An interval exchange with a rotational scheme is called a \emph{rotational interval exchange}.

The \emph{twist total} of an IRE scheme $\sigma$ (see~\cite{Teplinsky23}) is the number $T(\sigma)+T(\sigma^*)$, i.e., the sum of the twist numbers of schemes $\sigma$ and $\sigma^*$. In these terms, a \emph{rotational} scheme is an IRE scheme that is positive together with its dual and has zero twist total, while a \emph{rotational interval exchange} is a positive IRE with such a scheme. Accordingly to Sect.~10 of~\cite{Teplinsky23}, these are exactly those schemes, whose positive natural extensions generate translation surfaces of genus $g=1$, i.e., 2-d tori with no singular points.

It is clear that the class of all rotational schemes is closed w.r.t.\ duality operation. But why did we name such interval exchanges and their schemes ``rotational''? It is because they are directly related to circle rotations. The exact statement on this relation is the main result of the current paper, and it will be stated later in this section.

A rotation of a circle of length $L>0$ by distance $M$ is the map given by
\begin{equation}\label{eq:turn}
R_{L,M}:a\mapsto a+M,\quad a\in\R/L\Z,
\end{equation}
where the factor space $\R/L\Z$ is the circle of length $L$ by itself. A circle rotation is called \emph{irrational}, if its \emph{rotation number} $\rho=\{M/L\}\not\in\Q$. Here $\{\,\cdot\,\}$ denotes fractional part of a real number. 

Alternatively, we interpret a circle rotation as its projection onto an arbitrarily chosen half-open segment $[x_0,x_0+L)$, $x_0\in\R$, i.e., a piece-wise linear map
\begin{equation}\label{eq:turn_alt}
R_{L,M}:x\mapsto x+M-\left[\frac{x+M-x_0}{L}\right]L,\quad x\in[x_0,x_0+L),
\end{equation}
which is a rotational (two) intervals exchange by itself. Here $[\,\cdot\,]$ denotes integer part of a real number.

We call an \emph{arc} any half-open segment on a circle. In what follows, we will consider unions of finite number of arcs and interpret them, in accordance with the alternative definition~(\ref{eq:turn_alt}), as a union of finite number of half-open segments of a real line. If such a \emph{union of arcs} (we omit the words ``finite number of'', considering only finite unions) does not cover the whole circle, then it is natural to choose $x_0$ as a projection to $\R$ of any circle point that does not belong to the interior of any arc; it is also natural to consider as separate segments in this union the maximal ones (i.e., if two arcs overlap or touch by endpoints, then their projections to $\R$ should be united in one segment). In the case, when a union of arcs is the whole circle, the endpoint $x_0$ can be taken arbitrarily and interpret this union of arcs as the whole real segment $[x_0,x_0+L)$.

Finally, we can formulate our main result. Notice, that the third statement of this theorem is the main one, while the first two are rather additional: the qualitative (discrete) data is more important than the quantitative (real) one, since the space of allowed lengths is determined by an IRE scheme, and not otherwise around. However, we have chosen to order three theorem statements like this due to the logic of their proof: the third one's proof is based on the proofs of the first two, and in part the third statement is their direct implication (see Sect.~\ref{sect:proof_3} for more details). The proof of this theorem is constructive in a sense that the existence of all mentioned objects is proved by their algorithmic construction. In particular, notice the canonical form of rotational interval exchange we introduce in Sect.~\ref{subsect:canonic}, which is related to the so-called dynamical partitions of a circle (see~\cite{KhaninTeplinsky13}) and as such is the most convenient ``intermediary'' between general rotational interval exchanges and circle rotations.

\begin{theorem}
\label{theorem:main}
1) For any irrational circle rotation, the first return map onto any subset, which is a union of arcs, is an irreducible rotational interval exchange.

2) For any irreducible rotational interval exchange, there exists a first return map for a circle rotation onto a union of arcs, which is shift equivalent to that interval exchange.

3) An irreducible interval exchange scheme is rotational, if and only if there exists a first return map for an irrational circle rotation onto a union of arcs, which is an interval exchange with that scheme.
\end{theorem}

\begin{remark}
The theorem is formulated for irreducible schemes and interval exchanges, and a (single) circle rotation corresponds to each of them. If a interval exchange scheme consists of several irreducible components, then a dynamical system splits into the same number of dynamically independent components, and a first return map should be considered for a union of that many circle rotations. Hence, the statements~2) and~3) of the theorem can be reformulated by omitting mentions of irreducibility, while replacing a circle rotation with a union of circle rotations (irrational circle rotations---in the statement~3), whose number equals the number of irreducible components of the scheme.
\end{remark}

\begin{remark}
It is worth of explanation, why in the statements~1) and~3) irrationality of a circle rotation is mentioned, while in the statement~2) it is not. This is because if an interval exchange is a first return map onto a union of arcs for an irrational circle rotation, then some small perturbation of parameters in such a system produces an interval exchange with the same scheme, which is a first return map onto a union of arcs for a rational circle rotation (close to the starting irrational one). Accordingly to this fact, every such scheme (and due to the third statement of the theorem, this is the case for all rotational schemes) allows both irrational and rational (in above-mentioned sense) interval exchanges. However, on the other hand, there exist interval exchange schemes that allow only rational interval exchanges, namely the schemes containing a chain of cycles of the form
$$
\left\{
\left[\begin{array}{c}
\alpha_1\\                                                                                                                                        
\alpha_2
\end{array}\right];\quad
\left[\begin{array}{c}
\alpha_2\\                                                                                                                                       
\alpha_3
\end{array}\right];\quad\dots;\quad
\left[\begin{array}{c}
\alpha_m\\                                                                                                                                       
\alpha_1
\end{array}\right]
\right\}$$ 
or a single cycle of the form $
\left[\begin{array}{c}
\alpha\\                                                                                                                                        
\alpha
\end{array}\right]$. This kind of periodic chains of cycles is characteristical for rational interval exchanges: an arc moves along a certain trajectory on the circle, and finally returns onto itself. A dual to the IRE scheme shown above contains a pair of cycles
$$
\left\{\left[\begin{array}{ccc}
\alpha_m & \dots & \alpha_1\\                                                                                                                                        
& \varnothing &
\end{array}\right];\quad 
\left[\begin{array}{ccc}                                                                                                                                        
& \varnothing &\\
\alpha_m & \dots & \alpha_1
\end{array}\right]
\right\}
$$
(here $\varnothing$ denotes the absence of elements in a row), and it is obviously not positive. Interval exchange schemes of this kind have negative twist total, and their natural extensions do not produce any translation surfaces. Because of these reasons, we do not call such schemes and interval exchanges with such schemes rotational, although such interval exchanges can be first return maps for (rational only!) circle rotations onto certain unions of arcs.
\end{remark}

\section{Proving the first statement of Theorem}
\label{sect:proof_1}

\subsection{First return maps are finite}

Let us remind what is a first return map for a dynamical system with descrete time to a subset of its phase space. Assume that such a dynamical system is given as a map $f:X\to X$, and a non-empty subset $\Gamma\subset X$ has the following property: for every point $x\in\Gamma$ there exists a positive integer $n$ such that $f^{n}(x)\in\Gamma$, i.e., the trajectory of the point $x$ returns to the set $\Gamma$ after certain time $n$. Denote by $n(x,f,\Gamma)$ the \emph{first return time} of $x$ due to the action of $f$ to $\Gamma$, i.e., the smallest of such numbers: $f^i(x)\not\in\Gamma$ for all $1\le i<n(x,f,\Gamma)$, and $f^{n(x,f,\Gamma)}(x)\in\Gamma$. The map $f_\Gamma:x\mapsto f^{n(x,f,\Gamma)}(x)$ is called the \emph{first return map} for $f$ to $\Gamma$. Obviously, this map defines on this subset a new, induced, dynamical system $f_\Gamma:\Gamma\to\Gamma$.

Accordingly to this definition, for a given number $n\ge1$, all the points $x\in\Gamma$ such that $n(x,f,\Gamma)=n$, form a set $\Gamma_n=\Gamma\cap f^{-n}(\Gamma)\backslash f^{-(n-1)}(\Gamma)\backslash\dots\backslash f^{-1}(\Gamma)$. There is a disjoint splitting  $\Gamma=\bigcup_{n=1}^{+\infty}\Gamma_n$, $\Gamma_n\cap\Gamma_m=\emptyset$ for $n\ne m$, and we have $f_\Gamma(x)=f^n(x)$ for every $x\in\Gamma_n$, $n\ge1$.

Let us show that in the case of a circle rotation $R=R_{M,L}$ the first return map to any union of a finite number of arcs $\Gamma$ is always \emph{finite}, i.e., the set of first return times of all points of this union $\{n(x,f,\Gamma)|x\in\Gamma\}$ is bounded. 

If the rotation number is rational, i.e., $\rho=M/L=p/q$, where $p$ and $q$ are mutually prime positive integers, then the equivalence $R^q(x)\equiv x$ holds, therefore $n(x,R,\Gamma)\le q$ for all $x$ of the circle, and thus the first return map for a rational circle rotation to any subset (not just to a union of arcs) is always defined and finite.

Now assume that the rotation number $\rho$ is irrational. In this case, a trajectory of any point $x$ is everywhere dense on the circle, therefore for every $\delta>0$ there exists a positive integer $n_0$ such that there is no circle arc of length $\delta$ free from the points of the finite segment of this trajectory $(R^i(x))_{i=1}^{n_0}$. Moreover, since all the trajectories of a circle rotation are equal in their spacing, this $n_0$ does not depend on $x$. Taking the length of the shortest arc from the union $\Gamma$ as $\delta$, we obtaion the bound $n(x,R,\Gamma)\le n_0$ for all $x$ on the circle.

We have shown that a first return map for $R$ to $\Gamma$ is indeed finite: $\Gamma=\bigcup_{n=1}^{n_0}\Gamma_n$   
for some finite $n_0$. And since both intersection and union of any two finite unions of arcs are finite union of arcs as well, then every set $\Gamma_n$ is a finite union of arcs too; on every its arc the first return map $R_\Gamma$ is a shift (by distance $nM$). Interpreting circle arcs as real line intervals inside a certain segment $[x_0,x_0+L)$, we obtain

\begin{proposition}\label{prop:finiteness}
For any circle rotation, the first return map onto any finite union of arcs is an interval exchange transformation.
\end{proposition}

If a circle rotation was irrational, then the obtained interval exchange is irreducible due to minimality of this rotation as a dynamical system.

\subsection{First return maps are rotational}
\label{subsect:rotationality}

Let us be given an irrational circle rotation~(\ref{eq:turn}), (\ref{eq:turn_alt}) and a finite union of arcs $\Gamma$, which we interpret as a set of pairwise disconnected half-open real line segments contained in $[x_0,x_0+L)$. By Proposition~\ref{prop:finiteness}, the first return map $R_\Gamma:\Gamma\to\Gamma$ is an interval exchange. We will show that the scheme of this interval exchange, seen as an IRE, is rotational by definition, i.e., that its dual IRE scheme is an interval exchange scheme as well.

In order to write down the IRE scheme for the interval exchange $R_\Gamma$, let us denote every interval in this exchange by its own symbol thus forming an alphabet $\A$. The set $\Gamma$ is, on the one hand, a union of all pairwise disconnected intervals $I_{\alpha\be}$, $\alpha\in\A$, and on the other hand it is a union of their (pairwise disconnected as well) images $I_{\alpha\en}=\Gamma(I_{\alpha\be})$, $\alpha\in\A$. Each disconnected segment in $\Gamma$ is also split, on the one hand, into several beginning intervals $I_{\alpha_1\be},\dots,I_{\alpha_m\be}$ (counting from left to right), and on the other hand into several ending intervals $I_{\beta_1\en},\dots,I_{\beta_n}\en$ (from left to right as well); in the scheme $\sigma$ this segment is coded by a cycle $(\alpha_1\be,\dots,\alpha_m\be,\beta_n\en,\dots,\beta_1\en)$, $\{\alpha_i\}_{i=1}^m,\{\beta_j\}_{j=1}^n\subset\A$, $n,m\ge1$. The total array of endpoints $\xect\in\R^{\bar\A}$, where $x_{\alpha\be}$ is the left endpoint of a beginning interval $I_{\alpha\be}$, and $x_{\beta\en}$ is the right endpoint of an ending interval $I_{\beta\en}$, $\alpha,\beta\in\A$, forms a real component of the IRE $(\sigma,\xect)$. 

All the endpoints of a given interval exchange are visibly divided into four types: type L (left) are left endpoints of disconnected segments in $\Gamma$ (in the example above such is the point $x_{\alpha_1\be}$); type R (right) are right endpoints of disconnected segments in $\Gamma$ (in the example above such is $x_{\beta_1\en}$); type MB (middle beginning) are points of connection between neighboring beginning intervals (in the example above such are $x_{\alpha_i\be}$, $i\ne1$); and type ME (middle ending) are points of connection between neighboring ending intervals (in the example above such are $x_{\beta_j\en}$, $j\ne1$). In the IRE view, these types relate not to the real numbers $x_{\bar\xi}$ themselves, but rather to the symbols attached: $\bar\xi\in\bar\A$ is type L, if $\bar\xi=\xi\be$ and $\sigma^{-1}(\bar\xi)=\eta\en$; type R, if $\bar\xi=\xi\en$ and $\sigma^{-1}(\bar\xi)=\eta\be$; type MB, if $\bar\xi=\xi\be$ and $\sigma^{-1}(\bar\xi)=\eta\be$; and type ME, if $\bar\xi=\xi\en$ and $\sigma^{-1}(\bar\xi)=\eta\en$, for some $\xi,\eta\in\A$. 

It is worth mentioning, that every disconnected segment in $\Gamma$ (a cycle in the permutation $\sigma$ corresponds to it) possesses exactly one type L endpoint and exactly one type R endpoint (they are actually its left and right endpoints respectively), while there may be any number of types MB and ME endpoints lying on it, including zero.

Consider an arbitrary type MB endpoint $x_{\xi_1\be}$ (so that $\sigma^{-1}(\xi_1\be)=\eta_1\be$, $\xi_1,\eta_1\in\A$) and follow dynamical trajectories of its left $x_{\xi_1\be}^-=(x_{\xi_1\be}-\eps,x_{\xi_1\be})$, $\eps\to0$, and right $x_{\xi_1\be}^+=(x_{\xi_1\be},x_{\xi_1\be}+\eps)$, $\eps\to0$, infinitesimal half-neighborhoods (it is obvious, that for small enough $\eps>0$, the consecutive images of such half-neighborhoods under action of $R_\Gamma$ do not cover any endpoints during long enough time, in particular they do not split into smaller intervals, therefore this infinitesimal consideration is well-defined). 

The interval $I_{\xi_1\be}$, for which $x_{\xi_1\be}$ is its left endpoint, is mapped by $R_\Gamma$ onto the interval $I_{\xi_1\en}$, and accordingly the right half-neighborhood $x_{\xi_1\be}^+$ is mapped onto a right half-neighborhood of the left endpoint of $I_{\xi_1\en}$, which is $x_{\sigma(\xi_1\en)}^+$. Now there are two cases: the latter endpoint is either type ME (so that $\sigma(\xi_1\en)=\eta_*\en$, $\eta_*\in\A$), or type L (so that $\sigma(\xi_1\en)=\xi_2\be$, $\xi_2\in\A$). In the first case we stop, while in the second case continue to follow the trajectory of the chosen half-neighborhood. Similarly to the initial step, $x_{\sigma(\xi_1\en)}^+=x_{\xi_2\en}^+$ is mapped onto $x_{\sigma(\xi_2\en)}^+$, and we get two cases again: either the endpoint $x_{\sigma(\xi_2\en)}$ is type ME (so that $\sigma(\xi_2\en)=\eta_*\en$, $\eta_*\in\A$), or type L (so that $\sigma(\xi_2\en)=\xi_3\be$, $\xi_3\in\A$). In the first case we stop, and in the second case continue to follow the trajectory. At some step in this algorithm, the process will stop, because the number of type L endpoints is finite, and the map $R_\Gamma$ does not have periodic trajectories due to irrationality of the circle rotation we began with. Therefore, after the stop we will obtain a sequence of symbols $\xi_1\be$, $\xi_2\be=\sigma(\xi_1\be)$, \dots, $\xi_m\be=\sigma(\xi_{m-1}\be)$, $\eta_*\en=\sigma(\xi_m\be)$ and the corresponding sequence of half-neighborhoods $x_{\xi_1\be}^+$, $x_{\xi_2\be}^+=R_\Gamma(x_{\xi_1\be}^+)$, \dots, $x_{\xi_m\be}^+=R_\Gamma(x_{\xi_{m-1}\be}^+)$, $x_{\eta_*\en}^+=R_\Gamma(x_{\xi_m\be}^+)$, where the endpoints $x_{\xi_1\be}$, $x_{\xi_2\be}$, \dots, $x_{\xi_m\be}$, $x_{\eta_*\en}$ are type MB, L, \dots, L, ME respectively; $\xi_1,\dots,\xi_m,\eta_*\in\A$, $m\ge1$.

Now look at the trajectory of the left half-neighborhood of the same starting type MB endpoint. The interval $I_{\eta_1\be}$, for which $x_{\xi_1\be}$ is its right endpoint, is mapped by $R_\Gamma$ onto the interval $I_{\eta_1\en}$, and accordingly the left half-neighborhood $x_{\xi_1\be}^-$ is mapped onto the left half-neighborhood of the right endpoint of $I_{\eta_1\en}$, which is $x_{\eta_1\en}^-$. Two cases here: either the latter endpoint is type ME (so that $\eta_1\en=\sigma(\xi_*\en)$, $\xi_*\in\A$), or type R (so that $\eta_1\en=\sigma(\eta_2\be)$, $\eta_2\in\A$). In the first case we stop; in the second case we continue to follow the trajectory of the chosen half-neighborhood. Similarly to the initial step, the interval $I_{\eta_2\be}$ with the right endpoint $x_{\eta_1\en}$, is mapped onto the interval $I_{\eta_2\en}$, and so $x_{\eta_1\en}^-$ is mapped onto $x_{\eta_2\en}^-$. Again two cases: either $x_{\eta_2\en}$ is type ME (so that $\eta_2\en=\sigma(\xi_*\en)$, $\xi_*\in\A$), or type R (so that $\eta_2\en=\sigma(\eta_3\be)$, $\eta_3\in\A$). In the first case we stop; in the second case we continue. At some step the process will stop, since $R_\Gamma$ has no periodic trajectories, and we will finally obtain a sequence of symbols $\xi_1\be$, $\eta_1\be=\sigma^{-1}(\xi_1\be)$, $\eta_2\be=\sigma^{-1}(\eta_1\en)$, \dots, $\eta_n\be=\sigma^{-1}(\eta_{n-1}\be)$ and the sequence of half-neighborhoods $x_{\xi_1\be}^-$, $x_{\eta_1\en}^-=R_\Gamma(x_{\xi_1\be}^-)$, $x_{\eta_2\en}^-=R_\Gamma(x_{\eta_1\en}^-)$ \dots, $x_{\eta_n\en}^-=R_\Gamma(x_{\eta_{n-1}\en}^-)$, where the endpoints $x_{\xi_1\be}$, $x_{\eta_1\en}$, \dots, $x_{\eta_n\en}$ are type MB, R, \dots, R, ME respectively; $\eta_1,\dots,\eta_n\in\A$, $n\ge1$.

What is left is to see that necessarily $x_{\eta_n\en}=x_{\eta_*\en}$, and therefore $\eta_*=\eta_n$. Whe is this? This is because $R_\Gamma$ is the first return map for a circle rotation to the union of arcs $\Gamma$. A circle rotation $R$ is a continuous map, and so the left and right half-neighborhoods of any point never get separated under its action. Therefore, at the moment when an inner point (such as a type MB endpoint) of the subset $\Gamma$ returns again as an inner point (such as a type ME endpoint) to the subset $\Gamma$ first time, its left and right half-neighborhoods meet each other again, although before that they could at some moments visit the boundary (as type L and R endpoints) of the subset $\Gamma$. Moreover, the total return time (in terms of the number of iterates of $R$) for two half-neighborhoods from their splitting at the point $x_{\xi_1\be}$ till reunion at the point $x_{\eta_n\en}$ is the same for both of them. This means the following. Recall that any beginning interval $I_{\alpha\be}$ of the exchange $R_\Gamma$ is included in a certain set $\Gamma_k$ of points returning to $\Gamma$ exactly after $k$ iterates of the circle rotation $R$, $1\le k\le k_0$; we may denote this time as $k_{\alpha}$, $\alpha\in\A$. Accordingly, for the trajectory of the right half-neighborhood $x_{\xi_1\be}^+$, the total time before it reaches the point $x_{\eta_n\en}$ in terms of iterates of $R$ equals $k_{\xi_1}+\dots+k_{\xi_m}$, and for the left half-neighborhood $x_{\xi_1\be}^-$ this time equals $k_{\eta_1}+\dots+k_{\eta_n}$. From what we shown above, these times are necessarily equal, thus we have $\sum_{i=1}^mk_{\xi_i}=\sum_{j=1}^nk_{\eta_j}$ for any type MB endpoint of an interval exchange $R_\Gamma$.

Now we can explicitly write down an interval exchange with an IRE scheme dual to $\sigma$. The dual scheme $\sigma^*$ consists of all the cycles $(\xi_1\be,\dots,\xi_m\be,\eta_n\en,\dots,\eta_1\en)$ constructed for all the type MB endpoints in the interval exchange $R_\Gamma$. All these cycles are disjoint, since an interval exchange is bijective. On the other hand, each element of $\bar\A$ belongs to one of these cycles, because we counted all type MB endpoints; the number of type ME endpoints is the same as that of MB; and a half-neighborhood of any type L or type R endpoint necessarily hits a type ME endpoint after a number of iterates, therefore we counted all other endpoints as well. It is easy to check that $\sigma^*$ is dual to $\sigma$ by the definition (\ref{eq:duality}), in view of the relations we obtained when investigating trajectories of half-neighborhoods of $x_{\xi_1\be}$. To show that the scheme $\sigma^*$ is positive, let us put $y_{\xi_i\be}=\sum_{s=1}^{i-1}k_{\xi_s}$, $1\le i\le m$, and $y_{\eta_j\en}=\sum_{t=1}^jk_{\eta_t}$, $1\le j\le n$, for every cycle like investigated above; this will determine a vector of endpoints $\yect$ allowed by the scheme $\sigma^*$ due to the equalities $\sum_{i=1}^mk_{\xi_i}=\sum_{j=1}^nk_{\eta_j}$. The corresponding vector of lengths consists of positive (and even integer) components $k_\xi$, $\xi\in\A$, and so an IRE $(\sigma^*,\yect)$ is an interval exchange indeed.

The statement~1 of Theorem~\ref{theorem:main} is proved. 

\begin{remark}
Proving this statement, for an interval exchange induced on a subset of a circle by its rotation, we constructed a dual dynamical system of interval exchange, with intervals of time being exchanged like spacial intervals.
\end{remark}

\section{Proving the second statement of Theorem}
\label{sect:proof_2}

In this Section, we prove the statement~2) of Theorem~\ref{theorem:main} by presenting an algorithm for construction of an irrational circle rotation and a union of arcs such that the corresponding first return map is shift equivalent to the given rotational interval exchange.

Thus, let us be given a rotational interval exchange $(\sigma,\xect)$. Since we are talking about shift equivalent interval exchanges anyway, it will suffice to consider a floating IRE $(\sigma,\vect)$ from the start, i.e., restrict ourselves by considering lengths of intervals and do not mind their endpoints.

To this IRE, we will consecutively apply two types of operations: the first type is a reverse step of induction (this will produce a new dynamical system, for which the initial one is a first return map), while the second type is merging two neighboring intervals into a single one (this will not change the dynamical system at all, but the number of exchanging intervals will decrease).

\subsection{The elementary steps of induction}
\label{subsect:induction}

Here we remind what are the four elementary steps of induction $\Pi_{\alpha\beta}^{\mathrm{rb}}$, $\Pi_{\alpha\beta}^{\mathrm{re}}$, $\Pi_{\alpha\beta}^{\mathrm{lb}}$, and $\Pi_{\alpha\beta}^{\mathrm{le}}$, which we defined in~\cite{Teplinsky23}. They are operations of transformation, which can be applied to IREs (and to interval exchanges, in particular), or separately to there schemes, generalizing the classical steps of Rauzy--Veech induction. According to their action, we call these four steps ``cropping a beginning interval on the right'', ``cropping an ending interval on the right'', ``cropping a beginning interval on the left'', and ``cropping the ending interval on the left'' respectively. The general formulas for the four elementary induction steps are given in Sect.~7 of the work~\cite{Teplinsky23}; it is also explained there that the steps $\Pi_{\alpha\beta}^{\mathrm{rb}}$ and $\Pi_{\alpha\beta}^{\mathrm{re}}$ are applicable to an IRE scheme $\sigma$ under condition $\sigma(\alpha\be)=\beta\en$, while the steps $\Pi_{\alpha\beta}^{\mathrm{lb}}$ and $\Pi_{\alpha\beta}^{\mathrm{le}}$ are applicable under condition $\sigma(\beta\en)=\alpha\be$. In terms of cycles in the permutation $\sigma$, these operations act as follows: the step $\Pi_{\alpha\beta}^{\mathrm{rb}}$ (the steps $\Pi_{\alpha\beta}^{\mathrm{re}}$, $\Pi_{\alpha\beta}^{\mathrm{lb}}$, $\Pi_{\alpha\beta}^{\mathrm{le}}$) moves the element $\beta\en$ (the elements $\alpha\be$, $\beta\en$, $\alpha\be$) from its current position into the position straight before $\alpha\en$ (straight after $\beta\be$, straight after $\alpha\en$, straight before $\beta\be$). In the vector of lengths $\vect$, a single component gets changed: the steps $\Pi_{\alpha\beta}^{\mathrm{rb}}$ and $\Pi_{\alpha\beta}^{\mathrm{lb}}$ subtract the value of $v_\beta$ from the component $v_\alpha$, while the steps $\Pi_{\alpha\beta}^{\mathrm{re}}$ and $\Pi_{\alpha\beta}^{\mathrm{le}}$ subtract the value of $v_\alpha$ from the component $v_\beta$.
The reverse steps $(\Pi_{\alpha\beta}^{\mathrm{rb}})^{-1}$, $(\Pi_{\alpha\beta}^{\mathrm{re}})^{-1}$, $(\Pi_{\alpha\beta}^{\mathrm{lb}})^{-1}$, or $(\Pi_{\alpha\beta}^{\mathrm{le}})^{-1}$, are applicable to an IRE scheme $\sigma$ under conditions $\sigma(\beta\en)=\alpha\en$,  $\sigma(\beta\be)=\alpha\be$,  $\sigma(\alpha\en)=\beta\en$, or $\sigma(\alpha\be)=\beta\be$ respectively. On the lengths: the reverse steps $(\Pi_{\alpha\beta}^{\mathrm{rb}})^{-1}$ and $(\Pi_{\alpha\beta}^{\mathrm{lb}})^{-1}$ add the value of $v_\beta$ to the component $v_\alpha$, while the steps $(\Pi_{\alpha\beta}^{\mathrm{re}})^{-1}$ and $(\Pi_{\alpha\beta}^{\mathrm{le}})^{-1}$ add the value of $v_\alpha$ to the component $v_\beta$.

Next, let us visualize how these steps work on (floating) interval exchanges using two-row notation for cycles in their schemes. The classical Rauzy--Veech induction works effectively the same, though it acts on a single segment and at its right end only, while we consider interval exchanges on multiple segments and apply induction steps at both ends.

To guarantee that an interval exchange stays an interval exchange after applying an induction step or a reverse induction step, we must assure (in addition to the conditions listed above) that all components of the vector of lengths stay positive, and that the twist total of the scheme stays zero (in other words, the twist number of every cycle must stay zero, so that that cycle stays ``two-row'', see the definition in Sect.~2).

Consider the step $\Pi_{\alpha\beta}^{\mathrm{rb}}$ in detail. The condition of its applicability $\sigma(\alpha\be)=\beta\en$ means that the beginning interval $I_{\alpha\be}$ and the ending interval $I_{\beta\en}$ lie at the right edge of a certain segment $J$, i.e., in two-row notation a certain cycle in the scheme has the form $\left[\begin{array}{cc}
\dots & \alpha \\
\dots & \beta                                                                                                                                           \end{array}\right]$.
Since $\Pi_{\alpha\beta}^{\mathrm{rb}}$ subtracts the length $v_\beta$ from the length $v_\alpha$ (cutting of the ending interval $I_{\beta\en}$ from the segment, therefore ``cropping the beginning interval $I_{\alpha\be}$ on the right'' by $v_\beta$), the resulting vector of lengths stays positive if and only if $v_\alpha>v_\beta$, i.e., $I_{\alpha\be}$ is longer than $I_{\beta\en}$. If this condition holds, then the interval $I_{\beta\en}$ cannot be the only ending one on the segment $J$, because the sum of lengths of the ending intervals on a segment is always equal to the sum of lengths of the beginning ones, and therefore some other ending interval $I_{\gamma\en}$ lies on $J$ straight to the left from $I_{\beta\en}$, i.e., $\sigma(\beta\en)=\gamma\en$ for some $\gamma\in\A$. In two-row notation, the scheme is transformed as follows (we show only the cycles engaged in transformation; all other cycles do not change):
$$
\Pi_{\alpha\beta}^{\mathrm{rb}}:
\left[\begin{array}{cc}
\dots & \alpha \\
\dots & \gamma\ \beta                                                                                                                                           \end{array}\right],
\left[\begin{array}{ccc}
& \dots & \\
\dots & \alpha & \dots                                                                                                                                           \end{array}\right]
\mapsto
\left[\begin{array}{cc}
\dots & \alpha \\
\dots & \gamma                                                                                                                                           \end{array}\right],
\left[\begin{array}{ccc}
& \dots & \\
\dots & \alpha\ \beta & \dots                                                                                                                                           \end{array}\right].
$$
(Here the first cycle corresponds to the segment $J$, while the second one corresponds to the segment containing the ending interval $I_{\alpha\en}$; the latter can be the same segment $J$, in which case we have
$$
\Pi_{\alpha\beta}^{\mathrm{rb}}:
\left[\begin{array}{cccc}
& \dots & & \alpha \\
\dots & \alpha & \dots & \gamma\ \beta                                                                                                                                           \end{array}\right]
\mapsto
\left[\begin{array}{cccc}
& \dots & & \alpha \\
\dots & \alpha\ \beta & \dots & \gamma                                                                                                                                           \end{array}\right].
$$
If $\gamma=\alpha$, then the scheme $\sigma$ does not change under the action of $\Pi_{\alpha\beta}^{\mathrm{rb}}$, and only the length $v_\alpha$ does change.) In all cases, under the action of $\Pi_{\alpha\beta}^{\mathrm{rb}}$, the segment $J$ gets cropped on the right by $v_\beta$, which is the length of interval $I_{\beta\en}$. The intervals $I_{\alpha\be}$ and $I_{\alpha\en}$ also get cropped on the right by $v_\beta$, and the interval $I_{\beta\en}$ moves to a new position straight to the right of the (newly cropped by its length) interval $I_{\alpha\en}$. We also see that in all cases all cycles stay untwisted (twist number zero), and this is true for all (straight!) induction steps.

It is easy to see that the dynamical system of interval exchange obtained as a result of the action of $\Pi_{\alpha\beta}^{\mathrm{rb}}$ on $(\sigma,\vect)$ described above is nothing else but the first return map for the starting dynamical system on the union of segments, where the segment $J$ was cropped on the right by $v_\beta$. Indeed, all the points in this union return to it after time~1 except for the points of the interval $I_{\beta\be}$, which return after time~2: first they get to the cut-away interval $I_{\beta\en}$ of the starting dynamical system, and then the map brings them to the interval $I_{\beta\en}$ in the new dynamical system (in the starting system, this interval was the right part of the interval $I_{\alpha\en}$).

The rest three induction steps act similarly, and similarly each their action results in obtaining a first return map for the starting dynamical system to a union of segments, one of which was cropped from the right or from the left. For visualization, let us write down their action on schemes like it was done above for the step $\Pi_{\alpha\beta}^{\mathrm{rb}}$, making a note that the step $\Pi_{\alpha\beta}^{\mathrm{lb}}$ transforms an interval exchange into an interval exchange if and only if $v_\alpha>v_\beta$, while the steps $\Pi_{\alpha\beta}^{\mathrm{re}}$ and $\Pi_{\alpha\beta}^{\mathrm{le}}$ do this if and only if $v_\alpha<v_\beta$:

$$
\Pi_{\alpha\beta}^{\mathrm{re}}:
\left[\begin{array}{cc}
\dots & \gamma\ \alpha \\
\dots & \beta                                                                                                                                           \end{array}\right],
\left[\begin{array}{ccc}
\dots & \beta & \dots \\
& \dots &                                                                                                               \end{array}\right]
\mapsto
\left[\begin{array}{cc}
\dots & \gamma \\
\dots & \beta                                                                                                                                           \end{array}\right],
\left[\begin{array}{ccc}
\dots & \beta\ \alpha & \dots \\
& \dots &                                                                                                                                           \end{array}\right],
$$

$$
\Pi_{\alpha\beta}^{\mathrm{lb}}:
\left[\begin{array}{cc}
\alpha & \dots \\
\beta\ \gamma & \dots                                                                                                                                          \end{array}\right],
\left[\begin{array}{ccc}
& \dots & \\
\dots & \alpha & \dots                                                                                                                                           \end{array}\right]
\mapsto
\left[\begin{array}{cc}
\alpha & \dots\\
\gamma & \dots                                                                                                                                           \end{array}\right],
\left[\begin{array}{ccc}
& \dots & \\
\dots & \beta\ \alpha & \dots                                                                                                                                           \end{array}\right],
$$

$$
\Pi_{\alpha\beta}^{\mathrm{le}}:
\left[\begin{array}{cc}
\alpha\ \gamma & \dots \\
\beta & \dots 
\end{array}\right],
\left[\begin{array}{ccc}
\dots & \beta & \dots\\
& \dots &                                                                                                                                           \end{array}\right]
\mapsto
\left[\begin{array}{cc}
\gamma & \dots \\
\beta & \dots                                                                                                                                           \end{array}\right],
\left[\begin{array}{ccc}
\dots & \alpha\ \beta & \dots\\
& \dots &
\end{array}\right].
$$

Each of these steps decreases the greater one among two lengths $v_\alpha$, $v_\beta$ by the lesser one and does not change all other lengths.

The reverse elementary steps of induction $(\Pi_{\alpha\beta}^{\mathrm{rb}})^{-1}$, $(\Pi_{\alpha\beta}^{\mathrm{re}})^{-1}$, $(\Pi_{\alpha\beta}^{\mathrm{lb}})^{-1}$, or $(\Pi_{\alpha\beta}^{\mathrm{le}})^{-1}$, are applicable to an IRE scheme $\sigma$ under conditions $\sigma(\beta\en)=\alpha\en$, $\sigma(\beta\be)=\alpha\be$, $\sigma(\alpha\en)=\beta\en$, or $\sigma(\alpha\be)=\beta\be$ respectively (see Proposition~3 from~\cite{Teplinsky23}). The scheme obtained from an interval exchange scheme $\sigma$ by applying the reverse induction steps $(\Pi_{\alpha\beta}^{\mathrm{rb}})^{-1}$, $(\Pi_{\alpha\beta}^{\mathrm{re}})^{-1}$, $(\Pi_{\alpha\beta}^{\mathrm{lb}})^{-1}$, or $(\Pi_{\alpha\beta}^{\mathrm{le}})^{-1}$, has zero twist total if and only if there exists an element $\gamma\in\A$ such that $\sigma(\alpha\be)=\gamma\en$, $\sigma^{-1}(\beta\en)=\gamma\be$, $\sigma^{-1}(\alpha\be)=\gamma\en$, or $\sigma(\beta\en)=\gamma\be$ respectively. Since action of the steps $(\Pi_{\alpha\beta}^{\mathrm{rb}})^{-1}$ and $(\Pi_{\alpha\beta}^{\mathrm{lb}})^{-1}$ on the real component of an IRE $\vect$ adds the length $v_\beta$ to $v_\alpha$, while action of the steps $(\Pi_{\alpha\beta}^{\mathrm{re}})^{-1}$ and $(\Pi_{\alpha\beta}^{\mathrm{le}})^{-1}$ adds the length $v_\alpha$ to $v_\beta$, positivity of the starting IRE implies positivity of the resulting IRE, as one of the lengths just increases.

\begin{proposition}
\label{prop:induction-rotation}
An interval exchange obtained from a rotational interval exchange by applying an elementary step of induction to it, is also rotational.
\end{proposition}

\begin{proof}
According to Proposition~6 from~\cite{Teplinsky23}, induction steps do not change the twist total of a scheme, therefore in this case it stays zero. According to Theorem~1 from~\cite{Teplinsky23}, while an IRE scheme is transformed under the action of an elementary step of induction, its dual scheme transforms under the action of a certain reverse step. And since the action of a reverse induction step on a positive IRE just increases one of the lengths by the value of another one, the dual scheme stays positive. Proposition is proved.
\end{proof}

The similar statement regarding reverse induction steps is, in general, false, as we show by the following
{\bf counterexample}.

Consider the scheme $\sigma=
\left\{\left[\begin{array}{ccc}
\gamma & \alpha & \delta\\
\delta & \beta &                                                                                                                                           \end{array}\right],\left[\begin{array}{cc}
\beta & \\
\alpha & \gamma                                                                                                                                           \end{array}\right]\right\}$. It has zero twist number and is positive: for example, the vector of lengths
$\vect=(v_\alpha,v_\beta,v_\gamma,v_\delta)=(1,2,1,1)$ is allowed. The dual scheme
$\sigma^*=
\left\{\left[\begin{array}{ccc}
\alpha & \beta\\
\gamma & \beta & \delta                                                                                                                                           \end{array}\right],\left[\begin{array}{cc}
\delta & \gamma \\
\alpha &                                                                                                                                           \end{array}\right]\right\}$
is also untwisted and positive: for example, the vector of lengths $\wect=(w_\alpha,w_\beta,w_\gamma,w_\delta)=(2,1,1,1)$ is allowed. Therefore, the IRE $(\sigma,\vect)$ is a rotational interval exchange. The reverse step of induction $(\Pi_{\gamma\alpha}^{\mathrm{le}})^{-1}$ is applicable, and the obtained scheme $\sigma'=
\left\{\left[\begin{array}{cc}
\alpha & \delta\\
\delta & \beta                                                                                                                                          \end{array}\right],\left[\begin{array}{cc}
\gamma & \beta \\
\alpha & \gamma                                                                                                                                           \end{array}\right]\right\}$ is again untwisted and positive, as the vector of lengths $\vect=(1,2,1,1)$ is transformed into $\vect'=(2,2,1,1)$. Therefore, the obtained IRE $(\sigma',\vect')=(\Pi_{\gamma\alpha}^{\mathrm{le}})^{-1}(\sigma,\vect)$ is an interval exchange. But it is not rotational: the dual to $\sigma'$ scheme $(\sigma')^*=\Pi_{\alpha\gamma}^{\mathrm{le}}\sigma^*=
\left\{\left[\begin{array}{ccc}
\beta & & \\
\gamma & \beta & \delta                                                                                                                                           \end{array}\right],\left[\begin{array}{ccc}
\delta & \alpha & \gamma \\
\alpha & &                                                                                                                                           \end{array}\right]\right\}$ still has zero twist number, but it is obviously not positive, since vectors of lengths $\uect=(u_\alpha,u_\beta,u_\gamma,u_\delta)$ allowed by it are determined by the condition $u_\gamma+u_\delta=0$.

\subsection{The operation of merging intervals}

If in an interval exchange there are two neighboring intervals, which are shifted by the same distance, then it is natural to merge them into a single interval, and the dynamical system will not change. Such a situation takes place when $\sigma(\alpha\be)=\beta\be$ and $\sigma(\beta\en)=\alpha\en$ for some $\alpha\ne\beta$. Define the operation of \emph{merging intervals} $\Sigma_{\alpha\beta}$ that is applicable to an IRE $(\sigma,\vect)$ under condition mentioned above and acts in the following way: the symbol $\beta$ is removed from the alphabet $\A$, the elements $\beta\be$ and $\beta\en$ are removed from the cycles they belonged to, and the length $v_\alpha$ increases by the length $v_\beta$.

Formally, $\Sigma_{\alpha\beta}(\sigma,\vect)=(\sigma',\vect')$, where $\sigma'$ is a permutation on the reduced double alphabet $\bar\A'=\A'\times\{\be,\en\}$, $\A'=\A\backslash\{\beta\}$, given by the following equalities: in the case $\sigma(\beta\be)\ne\beta\en$, they are $\sigma'(\alpha\be)=\sigma(\beta\be)$, $\sigma'(\sigma^{-1}(\beta\en))=\alpha\en$, and $\sigma'(\bar\xi)=\sigma(\bar\xi)$ for $\bar\xi\in\bar\A'\backslash\{\alpha\be,\sigma^{-1}(\beta\en)\}$, while in the case $\sigma(\beta\be)=\beta\en$ they are $\sigma(\alpha\be)=\alpha\en$, and $\sigma'(\bar\xi)=\sigma(\bar\xi)$ for $\bar\xi\in\bar\A'\backslash\{\alpha\be\}$; the new vector of lengths $\vect'\in\R^{\A'}$ is given by $v'_\alpha=v_\alpha+v_\beta$, and $v'_\xi=v_\xi$ for $\xi\in\A'\backslash\{\alpha\}$.

Notice, that for a rotational scheme $\sigma$ the case $\sigma(\beta\be)=\beta\en$ is not possible, because this equality implies $\sigma^*(\beta\en)=\beta\en$ for the dual scheme $\sigma^*$, which makes the latter non-positive: every vector of lengths $\wect$ allowed by the scheme $\sigma^*$ contains a component $w_\beta=0$.

The condition $\sigma(\alpha\be)=\beta\be$, $\sigma(\beta\en)=\alpha\en$ for applicability of $\Sigma_{\alpha\beta}$ to $\sigma$ is equivalent to the condition $\sigma^*(\alpha\en)=\beta\be$, $\sigma^*(\beta\be)=\alpha\en$ on the dual scheme $\sigma^*$. This means that $\sigma^*$ contains the two-element cycle $\left[\begin{array}{c}
\beta\\
\alpha \end{array}\right]$. Applying the operation of merging intervals $\Sigma_{\alpha\beta}$ to $\sigma$ removes this two-element cycle from $\sigma^*$ and replaces $\beta\en$ by $\alpha\en$ in the cycle in $\sigma^*$, which contained that element $\beta\en$.

\begin{proposition}
The operation of merging intervals, applied to a rotational interval exchange, leaves it rotational.
\end{proposition}

\begin{proof}
It is clear from the definition of this operation that the twist numbers of cycles do not change, though one (untwisted, as all of them) cycle in the dual scheme completely disappears; therefore the twist total stays zero.
One of the lengths increases by the value of another one, therefore the vector of lengths stays positive. For the dual scheme, before applying the operation, there existed an allowed positive vector of lengths $\wect$ satisfying $w_\alpha=w_\beta$, therefore the vector $\wect'$ with the same components (only with the component $w_\beta$ removed) is obviously positive and allowed by the scheme that is dual to the one obtained by applying the operation $\Sigma_{\alpha\beta}$. Proposition is proved.
\end{proof}

\subsection{The idea of algorithm}

Let us describe the algorithm of consecutive transformation of an arbitrary irreducible rotational interval exchange, which will lead eventually to construction of a first return map we look for in the statement~2) of Theorem~\ref{theorem:main}. It will be convenient for us to operate with dual schemes. The idea is the following: applying consecutively elementary steps of induction to one of the cycles in the dual scheme, we eventually reduce it to two-element, and then remove it from consideration by merging corresponding intervals. After carrying out this procedure for every cycle of the dual scheme, one by one, we eventually obtain a dual scheme that consists of a single cycle. At that point the algorithm will stop; we will analyze the resulting interval exchange (it will be very special) and show how to construct a circle rotation and a union of arcs, for which this interval exchange will be the first return map. 

It will be important to keep the dual scheme positive after every new transformation (i.e., do not allow the effect described in the counterexample at the end of Subsect.~\ref{subsect:induction} to happen); in this case, by Proposition~\ref{prop:induction-rotation}, that scheme will always stay rotational, and therefore the interval exchange we are transforming will stay rotational as well. 

Two lemmas are needed.

\subsection{Lemma on unsplittability}

Let us call an interval exchange scheme $\sigma$ \emph{splittable} (at the cycle $c_0$), if its alphabet $\A$ can be split into two non-empty sub-alphabets $\A_1$ and $\A_2$ ($\A_1\cup\A_2=\A$, $\A_1\ne\varnothing$, $\A_2\ne\varnothing$) in such a way that the following property holds: one of the cycles $c_0=(\bar\xi,\sigma(\bar\xi)),\dots,\sigma^k(\bar\xi))$, $\bar\xi\in\bar\A$, $k\ge1$, $\sigma^{k+1}(\bar\xi)=\bar\xi$, in the scheme $\sigma=\{c_0,c_1,\dots,c_n\}$, $n\ge0$, can be split into two non-empty arcs $a_1=(\bar\xi,\sigma(\bar\xi),\dots,\sigma^i(\bar\xi))$ and $a_2=(\sigma^{i+1}(\bar\xi),\dots,\sigma^k(\bar\xi))$, $0\le i<k$, and the rest of the cycles can be split into two sets $\tau_1=\{c_1,\dots,c_j\}$ and $\tau_2=\{c_{j+1},\dots,c_n\}$, $0\le j\le n$, such that all the elements of $a_1$ and cycles from $\tau_1$ belong to $\bar\A_1=\A_1\times\{\be,\en\}$, while all the elements of $a_2$ and cycles from $\tau_2$ belong to $\bar\A_2=\A_2\times\{\be,\en\}$. Otherwise, a scheme is called \emph{unsplittable}. 

It is easy to show that for an interval exchange with a splittable scheme, the sum of lengths of all the beginning intervals in the arc $a_1$ is equal to the sum of lengths of all the ending intervals in the arc $a_1$, and the same is true for the intervals in the arc $a_2$. This follows from the observation that such a property (the sum of lengths of all beginning intervals from some set equals the sum of lengths of all ending intervals from that set) holds for every particular cycle by Proposition~\ref{prop:v=v}, and therefore for the set of all intervals indexed by the elements of the cycles from $\tau_1$, on one hand; it also holds for the set of all intervals indexed by the elements of $\bar\A_1$, on the other hand, and these two sets differ exactly by the set of all intervals in the arc $a_1$. 

In particular, the property showed above implies that any of the arcs $a_1$ or $a_2$ in a splittable interval exchange scheme $\sigma$ cannot consist of beginning elements only or ending elements only, or else the sum of corresponding allowed lengths would be zero, and therefore the scheme would not be positive. Since the cycle $c_0$ has zero twist, then necessarily one of these arcs starts with an ending element and ends with a beginning element, while another one starts with a beginning element and ends with an ending element. Assume for definiteness that the first arc is $a_1$, and the second one is $a_2$, then we have $\bar\xi=\alpha\en$, $\sigma^i(\bar\xi))=\beta\be$, $\sigma^{i+1}(\bar\xi)=\gamma\be$, $\sigma^k(\bar\xi)=\delta\en$ for some $\alpha,\beta\in\A_1$ and $\delta,\gamma\in\A_2$. 

The splittability of a scheme can be seen most visually if one considers a corresponding floating interval exchange $(\sigma,\vect)$: it simply means that its segments can be fixed in positions such that every interval of thus obtained fixed interval exchange $(\sigma,\xect)$ is wholly contained either in the half-line $(-\infty,0)$, or in the half-line $[0,+\infty)$; the labels of all the intervals in the first set belong to $\A_1$, while the labels of all the intervals in the second set belong to $\A_2$; and exactly for one segment (that corresponds to the cycle $c_0$) the origin is its inner point, and $x_{\alpha\en}=x_{\gamma\be}=0$ for  $\alpha$ and $\gamma$ determined in the previous paragraph. Effectively, the whole set of intervals in $(\sigma,\xect)$ is split by the point zero into two arrays: the intervals lying to the left from zero are indexed by the elements of $\bar\A_1$, while the intervals to the right from zero are indexed by the elements of $\bar\A_2$.

\begin{lemma}\label{lemma:split}
If an interval exchange scheme is splittable, then it cannot be rotational.
\end{lemma}

\begin{proof}
Due to the above mentioned properties of a splittable scheme $\sigma$, the sets of elements $\bar\A_1$ and $\bar\A_2$ are connected in it exactly in two places, namely: there exist elements $\alpha,\beta\in\A_1$ and $\delta,\gamma\in\A_2$ such that $\sigma(\beta\be)=\gamma\be$, $\sigma(\delta\en)=\alpha\en$, while for all $\bar\xi\in\bar\A_1\backslash\{\beta\be\}$ and for all $\bar\eta\in\bar\A_2\backslash\{\delta\en\}$ we have $\sigma(\bar\xi)\in\bar\A_1$ and $\sigma(\bar\eta)\in\bar\A_2$. A corresponding property holds for the dual scheme $\sigma^*$ as well, namely: $\sigma^*(\beta\en)=\gamma\be$, $\sigma^*(\delta\be)=\alpha\en$, while for all $\bar\xi\in\bar\A_1\backslash\{\beta\en\}$ and for all $\bar\eta\in\bar\A_2\backslash\{\delta\be\}$ we have $\sigma^*(\bar\xi)\in\bar\A_1$ and $\sigma^*(\bar\eta)\in\bar\A_2$. 
Therefore, one of the cycles in the scheme $\sigma^*$ can be split into two non-empty arcs $a'_1=(\alpha\en,\dots,\beta\en)$ and $a'_2=(\gamma\be,\dots,\delta\be)$, and the rest of cycles can be split into two sets $\tau'_1$ and $\tau'_2$ such that all the elements of $a'_1$ and cycles from $\tau'_1$ belong to $\bar\A_1$, while all the elements of $a'_2$ and cycles from $\tau'_2$ belong to $\bar\A_2$.
Let $\wect=(w_\xi)_{\xi\in\A}$ be a vector of lengths allowed by the scheme $\sigma^*$. The sum of lengths of all the beginning intervals in any cycle equals the sum of lengths of all the ending intervals in that cycle, and the same is obviously true for a set of cycles. Therefore, we have the equality
$
\sum_{\xi:\,\xi\be\in\tau'_1}w_\xi=\sum_{\xi:\,\xi\en\in\tau'_1}w_\xi
$ 
(a little bit abusive, but visual notation $\xi\be\in\tau'_1$ here means that the element $\xi\be$ belongs to one of the cycles from the set $\tau'_1$, and similarly for $\xi\en$). Since $\sum_{\xi:\,\xi\be\in\bar\A_1}w_\xi=\sum_{\xi\in\A_1}w_\xi=\sum_{\xi:\,\xi\en\in\bar\A_1}w_\xi$, and the set of all elements of cycles from $\tau'_1$ differs from $\bar\A_1$ exactly by the set of all elements of the arc $a'_1$, by subtracting the former equality from the latter one we obtain $\sum_{\xi:\,\xi\be\in a'_1}w_\xi=\sum_{\xi:\,\xi\en\in a'_1}w_\xi$.

If we assume that the splittable interval exchange scheme $\sigma$ is rotational, then the dual scheme $\sigma^*$ must by positive and untwisted. But if all the cycles in $\sigma^*$ have zero twist,then the arc $a'_1$ that we determined above contains ending elements only, therefore $\sum_{\xi:\,\xi\en\in a'_1}w_\xi=0$, and the scheme $\sigma^*$ is not positive.
 
Lemma is proved.
\end{proof}

\subsection{Lemma on unequal lengths}

Given an IRE scheme $\sigma$, we say that the lengths $v_\alpha$ and $v_\beta$, $\alpha,\beta\in\A$, are \emph{equal with necessity}, if for every allowed vector of lengths $\vect$ the equality $v_\alpha=v_\beta$ holds.
It is clear that all relations between the lengths are determined by the equalities (\ref{eq:v=v}), but if the scheme is complex enough, then equality with necessity for a certain pair of lengths can be non-obvious.

Consider a situation when, for some interval exchange, the beginning interval $I_{\alpha\be}$ and the ending interval $I_{\beta\en}$ are both lying by the left or by the right end of the same segment $J$. In two-row notation, the scheme contains a cycle $c_0$, which has the form $\left[\begin{array}{cc}
\alpha & \dots\\
\beta & \dots                                                                                                                                          \end{array}\right]$ or $\left[\begin{array}{cc}
\dots & \alpha\\
\dots & \beta                                                                                                                                          \end{array}\right]$, i.e., $\sigma(\beta\en)=\alpha\be$ or $\sigma(\alpha\be)=\beta\en$ respectively, $\alpha,\beta\in\A$, and this cycle is not two-element. In general case, the lengths $v_\alpha$ and $v_\beta$ can be equal with necessity, but not in the case of a rotational interval exchange, as shows the next statement.

\begin{lemma}\label{lemma:equal}
If for an interval exchange scheme $\sigma$ we have $\sigma(\beta\en)=\alpha\be$ or $\sigma(\alpha\be)=\beta\en$ for some $\alpha,\beta\in\A$, and the corresponding cycle $c_0$ is not two-element, but the lengths $v_\alpha$ and $v_\beta$ are equal with necessity, then this scheme cannot be rotational.
\end{lemma}

\begin{proof}
We will consider the case of $\sigma(\beta\en)=\alpha\be$ (for the case of $\sigma(\alpha\be)=\beta\en$ the proof is similar). Effectively, we will prove that under listed above conditions the scheme $\sigma$ is splittable, and then according to Lemma~\ref{lemma:split} it cannot be rotational. Let us start noticing that in the case of $\alpha=\beta$ the scheme $\sigma$ is obviously splittable (at the cycle $c_0$, with sub-alphabets $\A_1=\{\alpha\}$ and $\A_2=\A\backslash\{\alpha\}$), so for the rest of this proof we assume that $\alpha\ne\beta$.

Let $\vect$ be a positive vector of lengths allowed by the scheme $\sigma$. If there exists a (looped) sequence of distinct from $\alpha$ labels $\beta_1=\beta,\beta_2,\dots,\beta_k,\beta_{k+1}=\beta$, $k\ge1$, such that $\beta_{i+1}\en\in c(\beta_i\be)$ for all $1\le i\le k$, then simultaneously increasing all the lengths $v_{\beta_i}$, $1\le i\le k$, by the same positive value leaves true all the equalities (\ref{eq:v=v}), and therefore the lengths $v_\alpha$ and $v_\beta$ are not equal with necessity. Since this contradicts to the assumed conditions, such a looped sequence does not exist.

Collect the set of cycles $\tau_1$ by the next algorithm. First, we include there the cycle $c_1=c(\beta\be)$, and then will be adding all the cycles $c(\gamma\be)$, where $\gamma\ne\alpha$, and an element $\gamma\en$ belongs to one of the cycles we have already added to $\tau_1$. Effectively, the set $\tau_1$ consists of cycles of the form $c(\beta_k\be)$ for every existing finite sequence of distinct from $\alpha$ labels $\beta_1=\beta,\beta_2,\dots,\beta_k$, $k\ge1$, such that $\beta_{i+1}\en\in c(\beta_i\be)$ for all $1\le i<k$. The cycle $c_0$ does not belong to the set $\tau_1$ according to the previous paragraph.

Consider the set $\A_0\subset\A$ of all labels $\gamma\ne\alpha$ such that $\gamma\en\in\tau_1$ (as before, this slightly abusive notation means that an element belongs to one of the cycles in the set $\tau_1$). Due to our construction, we have $\gamma\be\in\tau_1$ for all $\gamma\in\A_0$; since $c_0$ does not belong to the set $\tau_1$, we have $\beta\en\not\in\tau_1$ and $\alpha\be\not\in\tau_1$. If $\alpha\en\not\in\tau_1$, then the set of the labels of all ending elements of cycles in the set $\tau_1$ is $\A_0$, while the set of the labels of all beginning elements of those cycles includes the set $\A_0$ and contains (at least) one more label $\beta\not\in\A_0$, which contradicts to Proposition~\ref{prop:v=v} for a positive scheme. Therefore, we have $\alpha\en\in\tau_1$, and the set of the labels of all ending elements of cycles in the set $\tau_1$ is $\A_0\cup\{\alpha\}$, while the set of the labels of all beginning elements of those cycles includes the set $\A_0\cup\{\beta\}$, and in fact coincides with it due to Proposition~\ref{prop:v=v} and equality of the lengths $v_\alpha$ and $v_\beta$.

Hence we see that in the case $\alpha\ne\beta$ the interval exchange scheme $\sigma$ is splittable (at the cycle $c_0$, with sub-alphabets $\A_1=\A_0\cup\{\alpha,\beta\}$ and $\A_2=\A\backslash\A_1$) as well, and therefore is not rotational. Lemma is proved.
\end{proof}

\subsection{Realization of algorithm}

So, let us be given a floating rotational interval exchange $(\sigma,\vect)$. Look at the dual to $\sigma$ rotational interval exchange scheme $\sigma^*$. If it contains two-element cycles, then we apply to $(\sigma,\vect)$ appropriate operations of merging intervals, so that there will be no two-element cycles left in the dual scheme $\sigma^*$ afterwards. If there are more than one cycle in $\sigma^*$, then we choose one of them arbitrarily and will be applying to it elementary induction steps of the form $\Pi^{\mathrm{lb}}_{\alpha\beta}$ or $\Pi^{\mathrm{le}}_{\alpha\beta}$ consecutively, where $\alpha$ and $\beta$ are the leftmost labels in the two-row notation for this cycle $c_0=\left[\begin{array}{cc}
\alpha & \dots\\
\beta & \dots                                                                                                                                          \end{array}\right]$, the upper and the lower ones respectively (i.e., $\sigma^*(\beta\en)=\alpha\be$), until this cycle becomes two-element. The interval exchange $(\sigma,\vect)$ during this process will be transformed by the action of reverse induction steps $(\Pi^{\mathrm{re}}_{\alpha\beta})^{-1}$ or $(\Pi^{\mathrm{le}}_{\beta\alpha})^{-1}$ in accorance with Theorem~1 from~\cite{Teplinsky23}. We will ensure that the scheme $\sigma^*$ stays an interval exchange scheme (i.e., in this case, simply stays positive), the by Proposition~\ref{prop:induction-rotation} it will stay rotational, and the interval exchange $(\sigma,\vect)$ during the corresponding transformations will stay rotational too. Let us describe how we choose each time which one of the two steps $\Pi^{\mathrm{lb}}_{\alpha\beta}$ and $\Pi^{\mathrm{le}}_{\alpha\beta}$ will be applied to the dual scheme $\sigma^*$ (just for a case, notice that the labels $\alpha$ and $\beta$ will be changing accordingly to changes in the cycle $c_0$ being transformed; we simply do not wish to overcomplicate the system of notations, thus leaving unchanged the notations $(\sigma,\vect)$, $\sigma^*$, $c_0$, $\alpha$, $\beta$ for objects that in fact change as the algorithm proceeds). Note straight away that $\alpha\ne\beta$, or else $\sigma^*$ would not be rotational. Four situations are now possible.

{\bf Situation 1:} $\alpha\en\not\in c_0$ and $\beta\be\not\in c_0$. If this a two-element cycle, then by merging of the corresponding intervals in $(\sigma,\vect)$ we remove it and go for the next cycle. If not, then by Lemma~\ref{lemma:equal} there exists a positive vector of lengths $\wect$ allowed by the scheme $\sigma^*$ such that $w_\alpha\ne w_\beta$. If $w_\alpha>w_\beta$, the we will apply to $\sigma^*$ the step $\Pi^{\mathrm{lb}}_{\alpha\beta}$, and if $w_\alpha<w_\beta$, then the step $\Pi^{\mathrm{le}}_{\alpha\beta}$. This will left the scheme $\sigma^*$ positive and therefore rotational, decreasing the number of elements in the cycle $c_0$: in the first case the element $\beta\en$ will be moved to the cycle containing $\alpha\en$, while in the second case the element $\alpha\be$ will be moved to the cycle containing $\beta\be$.

{\bf Situation 2:} $\alpha\en\not\in c_0$ and $\beta\be\in c_0$. Choose any positive vector of lengths $\wect$ allowed by the scheme $\sigma^*$. Since the length $w_\beta$ enters exactly one equality of (\ref{eq:v=v}) written down for $(\sigma^*,\wect)$, and does this on both left and right, then there is no restrictions on its value whatsoever, therefore it can be replaced by any positive number lesser than $w_\alpha$ (for example, put $w_\beta=w_\alpha/2$). Thus, there is a positive vector of lengths $\wect$ allowed by the scheme $\sigma^*$ in which $w_\alpha>w_\beta$. We will apply to $\sigma^*$ the step $\Pi^{\mathrm{lb}}_{\alpha\beta}$; this will left the scheme $\sigma^*$ positive and therefore rotational, and move the element $\beta\en$ from the cycle $c_0$ to the cycle containing $\alpha\en$.

{\bf Situation 3:} $\alpha\en\in c_0$ and $\beta\be\not\in c_0$. This one is similar to Situation~2, but now it is allowed to arbitrarily change the value of the length $w_\alpha$ in a positive vector of lengths $\wect$ allowed by the scheme $\sigma^*$. In particular there exists such a vector satisfying $w_\alpha<w_\beta$, therefore applying the step $\Pi^{\mathrm{le}}_{\alpha\beta}$ will leave the scheme $\sigma^*$ positive and therefore rotational, and move the element $\alpha\be$ from the cycle $c_0$ to the cycle containing $\beta\be$.

{\bf Situation 4:} $\alpha\en\in c_0$ and $\beta\be\in c_0$. Look at the cycle $c_0$ in two-row notation. Let there are $m\ge1$ labels in its lower row to the left from $\alpha$ and $n\ge1$ labels in its upper row to the left from $\beta$, i.e., $c_0=
\left[\begin{array}{cccccc}
\alpha_1 & \alpha_2 & \dots & \alpha_n & \beta & \dots\\
\beta_1 & \beta_2 & \dots & \beta_m & \alpha & \dots
\end{array}\right]$, $\alpha_1=\alpha$, $\beta_1=\beta$. Like in Situations~2 and~3, the lengths $w_\alpha$ and $w_\beta$ in a positive vector of lengths $\wect$ allowed by the scheme $\sigma^*$ can be chosen arbitrarily, hence applying any one of the two steps $\Pi^{\mathrm{lb}}_{\alpha\beta}$ or $\Pi^{\mathrm{le}}_{\alpha\beta}$ will leave the scheme $\sigma^*$ rotational. However, the number of elements in the cycle $c_0$ will not decrease. Instead, after applying $\Pi^{\mathrm{lb}}_{\alpha\beta}$ (applying $\Pi^{\mathrm{le}}_{\alpha\beta}$) the labels in the lower row to the left from $\alpha\en$ (in the upper row to the left from $\beta\be$) will be cyclically rearranged: 

$$
\Pi^{\mathrm{lb}}_{\alpha\beta_1}:
\left[\begin{array}{cccccc}
\alpha &  & \dots &  &  & \dots\\                                                                                                                                        
\beta_1 & \beta_2 & \dots & \beta_m & \alpha & \dots
\end{array}\right]
\mapsto
\left[\begin{array}{cccccc}
\alpha &  & \dots &  &  & \dots\\                                                                                                                                        
\beta_2 & \dots & \beta_m & \beta_1 & \alpha & \dots
\end{array}\right],
$$

$$
\Pi^{\mathrm{le}}_{\alpha_1\beta}:
\left[\begin{array}{cccccc}
\alpha_1 & \alpha_2 & \dots & \alpha_n & \beta & \dots\\
\beta &  & \dots &  &  & \dots                                                                                                                                         \end{array}\right]
\mapsto
\left[\begin{array}{cccccc}
\alpha_2 & \dots & \alpha_n & \alpha_1 & \beta & \dots\\
\beta &  & \dots & &  & \dots                                                                                                                                         \end{array}\right].
$$

\noindent
If among the labels $\beta_i$, $1<i\le m$, there is one such that $\beta_i\be\not\in c_0$, then applying consecutively the induction steps $\Pi^{\mathrm{lb}}_{\alpha\beta_1},\dots,\Pi^{\mathrm{lb}}_{\alpha\beta_{i-1}}$, we will arrive at Situation~3; similarly, if among the labels $\alpha_j$, $1<j\le n$, there is one such that $\alpha_j\en\not\in c_0$, then applying consecutively the induction steps $\Pi^{\mathrm{le}}_{\alpha_1\beta},\dots,\Pi^{\mathrm{le}}_{\alpha_{j-1}\beta}$, we arrive at Situation~2. In both cases, but the next step the number of elements in the cycle $c_0$ will be decreased. Assume now that all the elements $\alpha_1\en,\dots,\alpha_n\en$ and $\beta_1\be,\dots,\beta_m\be$ belong to the cycle $c_0$. The sets $\{\alpha_2,\dots,\alpha_n\}$ and $\{\beta_2,\dots,\beta_m\}$ cannot coincide, or else the scheme $\sigma^*$ would be either splittable (by the cycle $c_0$) with the sub-alphabets $\A_1=\{\alpha_2,\dots,\alpha_n,\alpha,\beta\}$, $\A_2=\A\backslash\A_1$, and therefore not rotational by Lemma~\ref{lemma:split}, or reducible (this is if $c_0$ does not contain elements not included in $\A_1\times\{\mathrm{b},\mathrm{e}\}$). Since the scheme $\sigma^*$ is rotational and irreducible, at least one of the sets $\{\beta_2,\dots,\beta_m\}\backslash\{\alpha_2,\dots,\alpha_n\}$ and $\{\alpha_2,\dots,\alpha_n\}\backslash\{\beta_2,\dots,\beta_m\}$ is non-empty. If the first one is non-empty, i.e., there exists $\beta_i\not\in\{\alpha_2,\dots,\alpha_n\}$, $1<i\le m$, then applying consecutively the steps $\Pi^{\mathrm{lb}}_{\alpha\beta_1},\dots,\Pi^{\mathrm{lb}}_{\alpha\beta_{i-1}}$, we will arrive at Situation~4 with increased $n$ (now it will be the number of labels in the upper row to the left from $\beta_i$, which is greater than former $n$, since $\beta_i$ was located to the right of $\beta$). If the set $\{\beta_2,\dots,\beta_m\}\backslash\{\alpha_2,\dots,\alpha_n\}$ is empty, then non-empty is the set
$\{\alpha_2,\dots,\alpha_n\}\backslash\{\beta_2,\dots,\beta_m\}$, and we similarly arrive at Situation~4, but this time with increased $m$. Since $n$ and $m$ are bounded, this process cannot continue indefinitely, therefore we will necessarily eventually arrive at Situations~2 or~3.

Summarizing the described algorithm of actions for a chosen cycle $c_0$ with more than two elements, in all cases we will reduce the number of its elements by one after a finite number of induction steps. Continuing to apply this algorithm to this cycle, we eventually transform it in a two-element cycle. Then we apply the operation of merging corresponding intervals in the interval exchange $(\sigma,\vect)$, and the new scheme $\sigma^*$ will contain one cycle less than before.
 
Continuing to choose cycles in $\sigma^*$ and applying the described algorithm to them, we will get rid of them one by one, and stop at the point when the scheme $\sigma^*$ will consist of a single cycle. By our construction, the resulting interval exchange (obtained by consecutive application of respective reverse induction steps and operations of merging intervals to $(\sigma,\vect)$) is rotational, and as a dynamical system the starting interval exchange is a first return map to the respective segments for the resulting interval exchange.

\subsection{The canonical form of a rotational interval exchange}
\label{subsect:canonic}

So, we reached the situation when the dual scheme $\sigma^*$ consists of a single cycle. The special property of an interval exchange with such a scheme is that among all its endpoints there is only one type L and only one type R (recall their classification from Subsection~\ref{subsect:rotationality}), while all the rest are either type MB or type ME. Accordingly, for the interval exchange $(\sigma,\vect)$ among all its endpoints there is only one type MB and only one type ME, while all the rest are type L or type R. This means that we have only one place with $\sigma(\alpha\be)=\beta\be$ and only one place with $\sigma(\gamma\en)=\delta\en$, $\alpha,\beta,\gamma,\delta\in\A$ (among these four labels the only equalities possible are $\alpha=\gamma$ or $\beta=\delta$, and any other equality is impossible due to positivity of the schemes $\sigma$ and $\sigma^*$); in all the other places an ending element is followed by a beginning element, and an ending element is followed by a beginning one. 

Two cases are possible: if $\alpha\be$ and $\gamma\en$ belong to two different cycles in the permutation $\sigma$, then these cycles have the form $\left[\begin{array}{cc}
\alpha &  \beta\\                                                                                                                                        
\kappa &
\end{array}\right]$ and $\left[\begin{array}{cc}
\lambda &\\                                                                                                                                        
\delta & \gamma
\end{array}\right]$, $\lambda,\kappa\in\A$ (an equality $\lambda=\kappa$ is possible), and all the rest of cycles are two-element; in the opposite case we have a cycle $\left[\begin{array}{cc}
\alpha &  \beta\\                                                                                                                                        
\delta & \gamma
\end{array}\right]$, and all the rest are two-element again. 

In view of the relations (\ref{eq:v=v}) it is easy too see that positivity of the schemes $\sigma$ and $\sigma^*$ takes place only when the set of all two-element cycles in $\sigma$, in the first case, is split into three finite sequences
\begin{align*}
&\left[\begin{array}{c}
\alpha_{i+1}\\                                                                                                                                       
\alpha_i
\end{array}\right],\quad 1\le i<m,\qquad 
\text{where}\quad \alpha_1=\alpha,\quad \alpha_m=\gamma,\quad m\ge1, 
\\
&\left[\begin{array}{c}
\beta_{j+1}\\                                                                                                                                       
\beta_j
\end{array}\right],\quad 1\le j< n,\qquad 
\text{where}\quad \beta_1=\beta,\quad \beta_n=\delta,\quad n\ge1,
\\
&\left[\begin{array}{c}
\lambda_{k+1}\\                                                                                                                                       
\lambda_k
\end{array}\right],\quad 1\le k< s,\qquad 
\text{where}\quad \lambda_1=\lambda,\quad \lambda_s=\kappa,\quad s\ge1
\end{align*}
(any of these sequences can be empty); in the second case, only the first two sequences are present. The single cycle in the dual scheme $\sigma^*$ in the first case is written as\\ $\left[\begin{array}{ccccccccc}
\beta_1 & \dots & \beta_n & \lambda_1 & \dots & \lambda_s & \alpha_1 & \dots & \alpha_m\\                                                                                                                                       
\alpha_1 & \dots & \alpha_m & \lambda_1 & \dots & \lambda_s & \beta_1 & \dots & \beta_n
\end{array}\right]$, and in the second case as\\ $\left[\begin{array}{cccccc}
\beta_1 & \dots & \beta_n & \alpha_1 & \dots & \alpha_m\\                                                                                                                                       
\alpha_1 & \dots & \alpha_m & \beta_1 & \dots & \beta_n
\end{array}\right]$.

The equalities (\ref{eq:v=v}) imply, in particular, that $v_\alpha=v_\gamma$ and $v_\beta=v_\delta$ in both cases, and an additional relation $v_\lambda=v_\kappa=v_\alpha+v_\beta$ in the first case.

If the first case takes place (i.e., the single cycle in $\sigma^*$ has the form\\ $\left[\begin{array}{ccccccccc}
\beta_1 & \dots & \beta_n & \lambda_1 & \dots & \lambda_s & \alpha_1 & \dots & \alpha_m\\                                                                                                                                       
\alpha_1 & \dots & \alpha_m & \lambda_1 & \dots & \lambda_s & \beta_1 & \dots & \beta_n
\end{array}\right]$ with $s\ge1$), then we will reduce it to the second case by consecutively applying the induction steps $\Pi^{\mathrm{le}}_{\beta_1\alpha_1},\dots,\Pi^{\mathrm{le}}_{\beta_n\alpha_1}$ to $\sigma^*$. As it was in Situation~4 from previous Subsection, the labels to the left from $\alpha_1$ in the upper row rearrange cyclically (while the scheme stays positive, and therefore rotational, due to the absence of restrictions on lengths), finally forming the sequence $\lambda_1, \dots , \lambda_s, \beta_1, \dots, \beta_n$, which coincides with that in the lower row to the right from $\alpha_m$.

Hence, in all cases we arrive at a situation when the rotational scheme $\sigma$ has the \emph{canonical form}
\begin{equation}\label{eq:canon}
\sigma_{\mathrm{can}}=\left\{
\left[\begin{array}{cc}
\alpha_1 &  \beta_1\\                                                                                                                                        
\beta_n & \alpha_m
\end{array}\right];\quad
\left[\begin{array}{c}
\alpha_{i+1}\\                                                                                                                                       
\alpha_i
\end{array}\right], 1\le i<m;\quad
\left[\begin{array}{c}
\beta_{j+1}\\                                                                                                                                       
\beta_j
\end{array}\right], 1\le j< n
\right\}
\end{equation}
for a certain set of pairwise different labels $\alpha_1,\dots,\alpha_m,\beta_1,\dots,\beta_n$ and certain positive integers $m$ and $n$. Taken with an allowed positive vector of lengths $\vect=\vect_{\mathrm{can}}$, the rotational scheme in the canonical form constitutes a rotational interval exchange $(\sigma_{\mathrm{can}},\vect_{\mathrm{can}})$ in the \emph{canonical form}. The components of the vector of lengths satisfy the relations $v_{\alpha_1}=\dots=v_{\alpha_m}$ and $v_{\beta_1}=\dots=v_{\beta_m}$, and we will denote these two lengths simply as $v_{\alpha}$ and $v_{\beta}$ respectively (it is not impossible for them to be equal as well).

Summarizing the last two Subsections, we have constructively proved the following proposition.
\begin{proposition}\label{prop:canon}
An arbitrary irreducible rotational interval exchange can be transformed into the canonical form by consecutively applying a finite number of reverse elementary steps of induction and operations of merging intervals; moreover, at each step of this process, the transformed IRE stays an irreducible rotational interval exchange.
\end{proposition}

\begin{remark}
As the presented algorithm shows, in fact, for transforming a rotational interval exchange into the canonical form it is enough to restrict the use of reverse elementary steps of induction by two types only, namely $(\Pi^{\mathrm{le}}_{\alpha\beta})^{-1}$ and $(\Pi^{\mathrm{re}}_{\alpha\beta})^{-1}$. Similarly, it would be enough to use the other two types only, namely $(\Pi^{\mathrm{lb}}_{\alpha\beta})^{-1}$ and $(\Pi^{\mathrm{rb}}_{\alpha\beta})^{-1}$; in application to the dual rotational scheme, the elementary induction steps $\Pi^{\mathrm{re}}_{\alpha\beta}$ and $\Pi^{\mathrm{rb}}_{\beta\alpha}$ respectively correspond to these two reverse steps accordingly to Theorem~1 from~\cite{Teplinsky23}.
\end{remark}

\subsection{The construction on a circle}
\label{subsect:circle}

For the rotational interval exchange $(\sigma_{\mathrm{can}},\vect_{\mathrm{can}})$ in the canonical form (\ref{eq:canon}), which we obtained from the starting rotational interval exchange $(\sigma,\vect)$, let us take big enough integers $k_1$ and $k_2$ and construct a circle rotation $R_{L,M}$ with $M=v_\beta+k_2v_\alpha$ and $L=v_\alpha+k_1M$. We consider this circle rotation in its projection onto the segment $[-v_\alpha,k_1M)$ according to (\ref{eq:turn_alt}), with $x_0=-v_\alpha$, i.e, as the map
$$
R_{L,M}:x\mapsto\left\{
\begin{array}{ll}
x+M, & x\in[-v_\alpha,(k_1-1)M)\\
x+M-L, & x\in[(k_1-1)M,k_1M)
\end{array}
\right..
$$ 

Let us mark on $[-v_\alpha,k_1M)$ the points $a_i=R_{L,M}^i(0)$, $0\le i<q$, of the trajectory segment of length $q=1+k_1+k_2k_1$ starting with the point $a_0=0$, under the action of $R_{L,M}$. It is easy to check that these $1+k_1+k_2k_1$ points are ordered as follows, from left to right:
$a_{k_1}=-v_\alpha$, $a_0=0$, then the array of $k_2+1$ points
$a_{k_2k_1+1}=v_\beta$, $a_{(k_2-1)k_1+1}=v_\beta+v_\alpha$, \dots, $a_{k_1+1}=v_\beta+(k_2-1)v_\alpha$, $a_1=M$, and then consecutively $k_1-1$ more of such arrays, shifted by $1\le j<k_1$ rotations $R_{L,M}$, i.e., the arrays of $k_2+1$ points
$a_{k_2k_1+1+j}=v_\beta+jM$, $a_{(k_2-1)k_1+1+j}=v_\beta+v_\alpha+jM$, \dots, $a_{k_1+1+j}=v_\beta+(k_2-1)v_\alpha+jM$, $a_{1+j}=(1+j)M$, where in the last array (for $j=k_1-1$) the last point $a_{1+(k_1-1)}=k_1M$ is the right endpoint of the segment $[-v_\alpha,k_1M)$, which in projection coincides with its left endpoint $a_{k_1}=-v_\alpha$ we already included in the list. Also note that $a_q=R_{L,M}^q(0)=v_\beta-v_\alpha$.

In view of this order and the equalities $a_{k_1}=-v_\alpha$, $a_0=0$, $a_{k_2k_1+1}=v_\beta$ it is easy to see that the points $a_i$, $0\le i<q$, split the circle $[-v_\alpha,k_1M)$ into $k_1$ arcs of length $v_\beta$, namely the arcs $R_{L,M}^i[0,v_\beta)$, $0\le i<k_1$, and $k_2k_1+1$ arcs of length $v_\alpha$, namely the arcs $R_{L,M}^i[-v_\alpha,0)$, $0\le i<k_2k_1+1$. Any two of these $1+k_1+k_2k_1$ arcs do not overlap, and there union is the whole circle. The arcs $R_{L,M}^{k_1}[0,v_\beta)=[a_{k_1},a_q)=[-v_\alpha,v_\beta-v_\alpha)$ and $R_{L,M}^{k_2k_1+1}[-v_\alpha,0)=[a_q,a_{k_2k_1+1})=[v_\beta-v_\alpha,v_\beta)$ also do not overlap and there union is the arc $[-v_\alpha,v_\beta)$, which is also the union of the arcs $[0,v_\beta)$ and $[-v_\alpha,0)$. 

Having in mind this construction on the circle, let us choose the arc $[-v_\alpha,v_\beta)$, any $m-1$ arcs among $R_{L,M}^i[-v_\alpha,0)$, $0\le i<k_2k_1+1$, and any $n-1$ arcs among $R_{L,M}^i[0,v_\beta)$, $0\le i<k_1$, in such a way that no two arcs would touch by there endpoints (this is clearly possible as soon as $k_1$ and $k_2$ are big enough). In accordance with our construction, the dynamical system determined by the first return map for the circle rotation $R_{L,M}$ to the union of chosen $n+m-1$ arcs is identical to the dynamical system of the rotational interval exchange in the canonical form $(\sigma_{\mathrm{can}},\vect_{\mathrm{can}})$. And since the latter interval exchange in the canonical form was obtained by consecutive application of reverse induction steps and operations of merging intervals towards the starting rotational interval exchange $(\sigma,\vect)$, then the dynamical system of the last one is, in its turn, determined by a first return map to a certain finite union of segments in the phase space of the dynamical system $(\sigma_{\mathrm{can}},\vect_{\mathrm{can}})$. Choosing on the circle the union of arcs that corresponds to this union of segments, we obtain the finite union of arcs such that the first return map to it for the circle rotation $R_{L,M}$ is shift equivalent to the starting irreducible rotational interval exchange $(\sigma,\xect)$.

The statement~2 of Theorem~\ref{theorem:main} is proved.

\section{Proving the third statement of Theorem}
\label{sect:proof_3}

In the third part of Theorem we formulated a criterion for an interval exchange scheme to be rotational in terms of a first return map to a union of arcs for an irrational circle rotation. Effectively, this statement is almost implied by the first two statements of Theorem, which we already proved, namely: the first statement implies that if the mentioned first return map exists, then the corresponding irreducible IRE scheme is rotational; the second statement implies that for an irreducible rotational scheme there exists a first return map with this scheme. The only thing left to be proved is that for an irreducible rotational scheme there exists a corresponding first return map for an irrational circle rotation namely. To do this now, we will return to the algorithm of transforming a rotational interval exchange to the canonical form that we used in the previous Section.

So let us be given an irreducible rotational interval exchange scheme $\sigma$. By choosing an arbitrary allowed vector of lengths $\vect$ we obtain a rotational interval exchange $(\sigma,\vect)$. Transform it into the canonical form $(\sigma_{\mathrm{can}},\vect_{\mathrm{can}})$ according to Proposition~\ref{prop:canon}. It is shown in Subsection~\ref{subsect:circle} that the canonical interval exchange with lengths $v_\alpha$ and $v_\beta$ is a first return map for the circle rotation $R_{L,M}$ with $M=v_\beta+k_2v_\alpha$ and $L=v_\alpha+k_1M$ with certain positive integers $k_1$ and $k_2$. The rotation number of this circle rotation $\rho=M/L=1/(k_1+1/(k_2+\rho_0))$, where $\rho_0=v_\beta/v_\alpha$, is either rational, or irrational, depending on rationality or irrationality of the number $\rho_0$. Hence, if the lengths $v_\alpha$ and $v_\beta$ are incommensurable, then the necessary irrational circle rotation is already constructed. 

Assume that the lengths $v_\alpha$ and $v_\beta$ are commensurable, i.e., that $\rho_0=v_\beta/v_\alpha\in\Q$. In this case we simply change them by small perturbation in such a way that they stop to be commensurable, for ex., by replacing the lengths $v_\beta$ by $v_\beta'=v_\beta+\varepsilon$, where $0<\varepsilon\ll 1$, $\varepsilon/v_\beta\not\in\Q$. Now let us conduct the whole algorithm of transformation into the canonical form, but backwards. The interval exchange $(\sigma_{\mathrm{can}},\vect_{\mathrm{can}})$ was obtained from $(\sigma,\vect)$ by applying a finite number of reverse steps of induction and operations of merging intervals. Accordingly, let us apply in backward order consecutively the corresponding induction steps and operations of splitting intervals (into determined parts). Obviously, each of these operations is robust in the sense that, after its application, small perturbations in real components (lengths) of an interval exchange stay small, and therefore discrete components (schemes) stay unchanged. Hence, starting with the perturbed as we said canonical interval exchange $(\sigma_{\mathrm{can}},\vect'_{\mathrm{can}})$ and applying the algorithm of transformation in backward order, we obtain a perturbed interval exchange $(\sigma,\vect')$ with the same given scheme $\sigma$, and this perturbed interval exchange is the first return map for now irrational circle rotation with rotation number $\rho'=1/(k_1+1/(k_2+\rho'_0))$, where $\rho'_0=(v_\beta+\varepsilon)/v_\alpha\not\in\Q$.

The statement~3 of Theorem~\ref{theorem:main} is proved. 

This work was financially supported in part by the National Academy of Sciences of Ukraine project ``Mathematical modeling of complex dynamical systems and processes related to security of the state'', Reg.~no.\ 0123U100853.

The author declares the absence of a conflict of interests.

\end{document}